\documentclass[12pt,a4paper]{amsart}
\usepackage[latin1]{inputenc}
\usepackage{amsmath, amsfonts, amssymb, amsthm, mathtools, graphicx, braket, subcaption}

\usepackage{tikz}
\usetikzlibrary{math} 
\usepackage{braids}
\usepackage[numbered]{hypdoc}
\definecolor{lstbgcolor}{rgb}{0.9,0.9,0.9}

\newcommand{\circdist}{1}  
\newcommand{\circrad}{7/4} 
\newcommand{\circlethickness}{2mm} 

\pgfmathsetmacro{\intrad}{sqrt((\circrad)^2 - 3*(\circdist)^2/4) - \circdist/2}

\pgfmathsetmacro{\extrad}{sqrt((\circrad)^2 - 3*(\circdist)^2/4) + \circdist/2}

\colorlet{180}{blue}
\colorlet{60}{red}
\colorlet{300}{green}

\newcommand{\mycircle}[1]{%
  \draw[thick, double distance=\circlethickness, double=#1]
  (#1:\circdist) circle (\circrad);}

\title{Topological invariants of sorting networks.}
\author{Maxim Arnold \and Christian Kondor}
\email{maxim.arnold@utdallas.edu, christian.kondor@utdallas.edu }
\address{The University of Texas at Dallas. 800 W. Campbell Rd., Richardson, TX 75080}

\newtheorem{thm}{Theorem}

\newtheorem{prp}{Proposition}
\newtheorem{lm}{Lemma}

\theoremstyle{definition}
\newtheorem{dfn}{Definition}
\newtheorem{rem}{Remark}

\newtheorem{qst}{Question}
\newtheorem{exm}{Example}

\newcommand{\sort}{\mathcal{S}}
\newcommand{\asb}{\mathfrak{A}}
\newcommand{\dsb}{\mathfrak{D}}

\newcommand{\lk}{\mathsf{lk}}

\newcommand{\R}{\mathbb{R}}

\newcommand{\symmetricn}{\mathfrak{S}_n}
\newcommand{\symmetric}{\mathfrak{S}}

\begin{document}
\maketitle

\begin{abstract}
	In this note we investigate finite-type Vassiliev invariants of the pure braids arising from signed sorting networks.

\end{abstract}

\section{Introduction}
For a given $n$, a sorting network on $n$ elements is a representation of the reverse permutation
$$\begin{pmatrix}
1&\dots&n\\
n&\dots&1
\end{pmatrix}$$
as a product of elementary transpositions ${\bf j}:=(j,j+1)$, $j=1,\dots,(n-1)$. This product must be of minimal length, being $\binom{n}{2}$ transpositions in length. We will denote this ``length of a sorting network'' by $N$.

Sorting networks have been extensively studied in the past few decades from the viewpoints of both computer science and pure mathematics. These studies have been motivated by sorting networks' rich combinatorial structure and wide range of applications (see e.g. \cite{Romik2007}, \cite{Virag2018} and references therein). One such application is to view a sorting network as a discretization of the generalized $1$-dimensional Euler equation (see for example \cite{Brennier}). In this view, physical space is discretized to a finite number $n$ of particles, the incompressibility condition is interpreted as the conservation of this $n$, and the action to be minimized is the squared discrete velocity \cite{Virag2016}. Our aim in this paper is to investigate asymptotic topological invariants of the trajectories in this limited case. We hope that our study will add to the discussion regarding the question posed by Arnold \cite{Arnold}.

\section{Basic Definitions}
This paper concerns the symmetric group $\symmetric_n$ on $n$ elements, since $\symmetricn$ is generated by the set $\{{\bf j}:=(j,j+1), \quad j=1,\dots,n-1\}$ of elementary transpositions. We will often refer to $\symmetricn$ in terms of its Cayley graph, the $n$-permutahedron. To construct the $n$-permutahedron, consider $\mathbf{x} = (x_1, \cdots, x_j, \cdots, x_n)$ and a central arrangement of $(n-1)$-dimensional hyperplanes $p_{j,k}:x_j=x_{k}$ in $\R^n$. These hyperplanes split the whole of $\R^n$ into $n!$ disjoint domains. If from one of these domains we choose the base point $\mathbf{b} = (1, 2, \cdots, n)$, we can associate each domain to the image of $\mathbf{b}$ under reflections through the corresponding hyperplanes. Letting these images be vertices and representing hyperplane reflections as edges connecting the images, we construct the graph known as the $n$-permutahedron. Labelling the vertices of the $n$-permutahedron with the inverse permutations coming from the components of the vertex provide a labelling on the edges. This edge labelling corresponds to the generators $\bf j$. For example, a vertex with coordinates $(3,1,2)$ would be labelled with the permutation $$\begin{pmatrix}1&2&3\\2&3&1\end{pmatrix}=(2,3)(1,2).$$

Sorting networks correspond to paths on the $n$-permutahedron. In particular, a sorting network on $n$ elements corresponds to a minimal-length path on the $n$-permutahedron that connects the identity permutation to the so-called ``long element'' of the symmetric group, as in figure \ref{fig:1B}.

\begin{figure}[hbt]
	\includegraphics[width=0.85\textwidth]{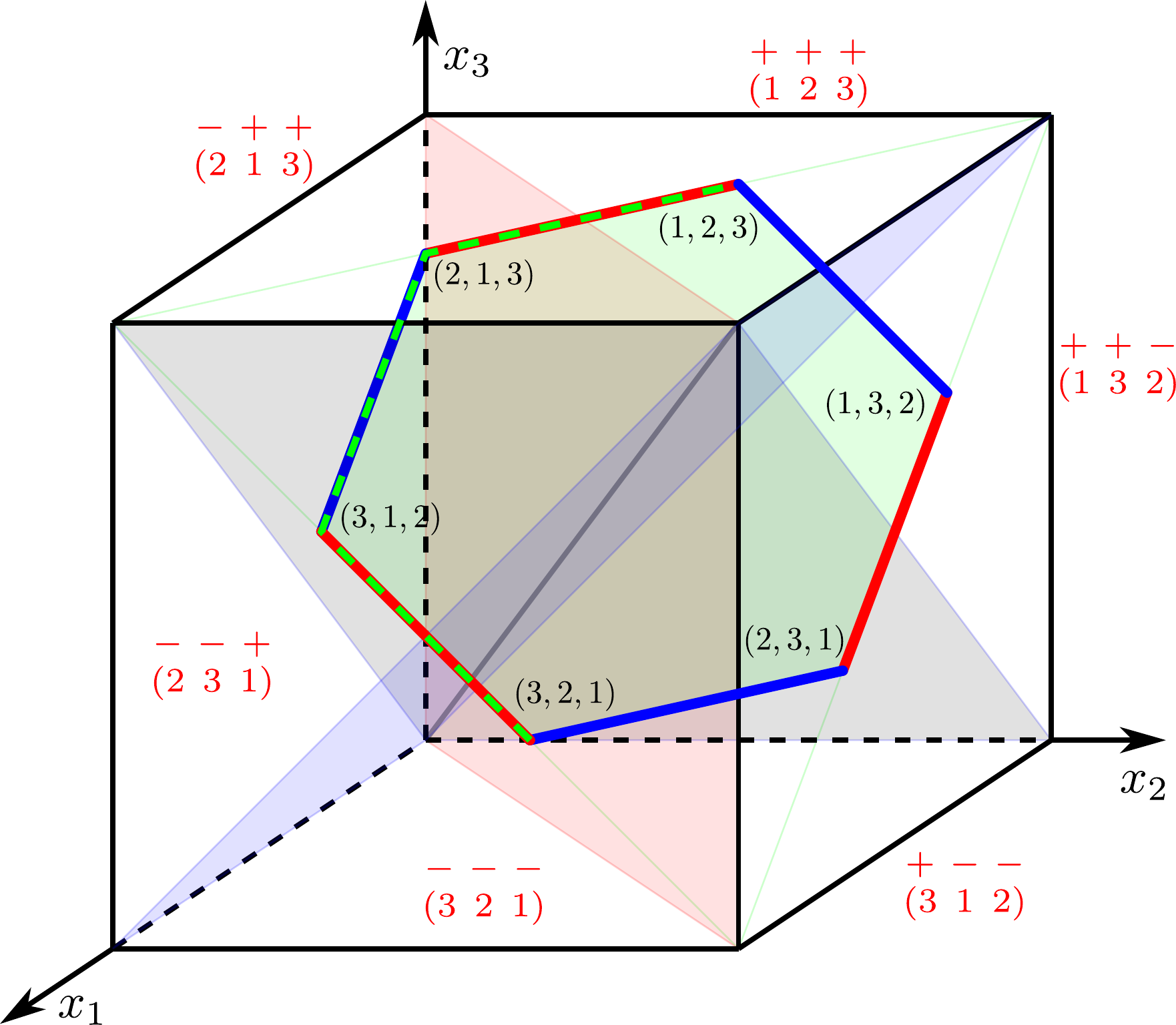}
	\caption{Permutahedron for $\symmetric(3)$. Planes $p_{1,2}$, $p_{2,3}$ and $p_{1,3}$ form central hyperplane arrangement and split $\mathbb{R}^3$ into six chambers. Every chamber corresponds to one of the orderings of components $(x_1,x_2,x_3)$. The sorting network $\bf 121$ corresponds to the path from the chamber $(+++):=\{x_2>x_1,\, x_3>x_1,\, x_3>x_2\}$ to the chamber $(---):=\{x_2<x_1,\,x_3<x_1,\,x_3<x_2\}$.}\label{fig:1B}
 \end{figure}

This, however, is not the only geometric representation of sorting networks. To see another representation consider the sorting network $S={\bf j}_1\cdots{\bf j}_N$ acting on $n$ elements, where $N = \binom{n}{2}$ as before. We can correspond this network to the trajectories of $n$ particles in a $1$-dimensional dynamical system during finite time $N$. 
The trajectory of the $i$-th particle, then, is the the sequence $s_k^{-1}(i)$ of locations of the letter $i$ in the configuration $s_k={\bf j}_1\cdots {\bf j}_k$. The set of all trajectories taken together forms the \textit{wiring diagram} of the sorting network\footnote{Throughout the paper we will denote the sorting network as a sequence of elementary transpositions $\mathbf{j}_k$ written in the order of wiring diagram, which is opposite to the order of the factors if we write the corresponding product. For example, the product of transpositions $(2,3)(1,2)$, providing the permutation $(2,3,1)$ would be denoted as $\bf 12$.}.

This wiring diagram lends itself to further geometric constructions. If we assign a signature $+$ or $-$ to each transposition (or ``crossing'') in a wiring diagram, we can construct a braid from a sorting network. This braid construction can be taken a step further: if we consider two sorting networks $S={\bf j}_1\cdots {\bf j}_N$ and  $T={\bf k}_1\cdots {\bf k}_N$, the word $ST={\bf j}_1\cdots {\bf j}_N{\bf k}_1\cdots {\bf k}_N$ 
corresponds to a closed loop on the permutahedron tracing a path from the identity to the long element along the word $S$ and then returning back to the identity along the word $T$.  We will call such a loop a \emph{sorting loop}. If as before we assign signatures to each crossing in the wiring diagram for $TS$, we create a pure braid on $n$ elements as in figure \ref{fig:sorting_braid}.

\subsection{The Free Group on Strands, Finite-Type Invariants.}
Constructing pure braids from sorting networks lends us a number of new tools for analyzing sorting networks and distinguishing them from one another. One subset of these tools is the finite-type invariants, which may be used to classify braids, knots, and links (see \cite{PolyakViro} and \cite{Braids}). An invariant of particular interest is the \textit{Milnor number}, sometimes referred to as the generalized linking number of a braid.

We can take the closure of a pure braid -- that is, connecting the ``starting'' and ``terminal'' points of each strand -- to form a closed link of $n$ loops embedded in $\R^3$. If all of the loops in such a link can be separated without any cuts, the link and its corresponding pure braid are called trivial. Two loops, however, may be interlinked. This interlinking results from two strands of a braid being wrapped around one another. How interlinked two strands $i$ and $j$ are can be quantified by their linking number $\lk(i, j)$, a well-known invariant of braids. The number $\lk(i,j)$ is the number of times that $i$ crosses over $j$ minus the number of times $j$ crosses over $i$.

\begin{figure}[!hbt]\centering
		\begin{tikzpicture}[scale=.75]
  \foreach \angle in {180, 60, 300}
  {
    \mycircle{\angle}
  }

  \foreach \angle/\rad in {60/\intrad, 240/\extrad}
  {
    \begin{scope}
      \clip (\angle:\rad) circle (10/4*\circlethickness);
      \mycircle{300}
      \mycircle{180}
    \end{scope}
  }
\end{tikzpicture}
	\caption{The Borromean rings, a non-trivial 3-link with all pairwise linking numbers equal to zero.}\label{fig:borromean}
\end{figure}
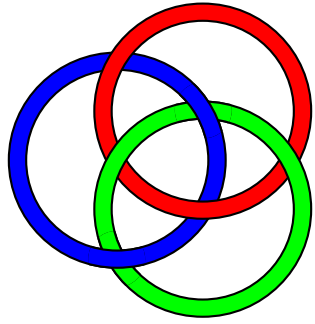

Intuition might lead us to believe that if the vector $\{\lk(i,j)\}\in \mathbb{Z}^N$ of the linking numbers of each pair of trajectories is the zero vector, then the braid is trivial, but this is certainly not the case! The most famous example of a non-trivial $3$-link having all pairwise linking numbers equal to zero is the Borromean rings, shown in figure \ref{fig:borromean}. Thankfully, we can generalize the notion of a linking number to describe the entanglement of larger groups of strands and detect finer structures. The generalization to $m$ strands is known as the \textit{$m$-th Milnor number}, defined as follows:

\begin{dfn}[see e.g. \cite{Braids}] The $m$-th Milnor number or \emph{generalized linking number} is defined for the pure braid $\beta$ on $m$ strands $\{p_1,\dots,p_m\}$. Let $w$ be a word corresponding to the braid $\beta$ in the free group presentation. Then $w$ can be written as $S_{p_m} p_m S_{p_m}^{-1}$. For any subsequence $d=t_{j_1}^{\epsilon_{j_1}}\cdots t_{j_{m-1}}^{\epsilon_{j_{m-1}}}$ inside $S_{p_m}$ such that $j_1<\dots<j_{m-1}$ and $t_{j_1}=p_1$, $\dots, t_{j_{m-1}}=p_{m-1}$, define the sign of the subsequence as $\mathsf{sgn}(d)=\epsilon_{j_1}\cdots \epsilon_{j_{m-1}}$. Then the $m$-th Milnor number for the braid $\beta$ is defined as the sum of $\mathsf{sgn}(d)$ over all sub-sequences of $S_{j_m}$.
\end{dfn}
In particular, $\mu_2(i,j)=\lk(i,j)$.
Informally, $\mu_m$ detects whether a set of loops forms a Brunnian link, a structure in which $m$ loops are linked, such that any $(m-1)$ of them are unlinked. For example, for the Borromean rings, $\mu_2(1,2)=\mu_2(2,3)=\mu_2(1,3)=0$ and $\mu_3(1,2,3)=1$.

\section{The Case of $n=3$.}\label{sec:3}
Consider the case of $n=3$ particles. The symmetric group $\symmetric_3$ has $6$ elements as illustrated in Fig. \ref{fig:1B}. The Cayley graph of $\symmetric_3$ is therefore a regular hexagon contained by the plane $x_1+x_2+x_3=1+2+3$. On this graph, the set of the reduced words for the long element is generated by ${\bf 1}=(1,2)$ and ${\bf 2}=(2,3)$, corresponding to reflections in the planes $p_{1,2}:\, x_1-x_2=0$ and $p_{2,3}:\, x_2-x_3=0$, respectively. These two transpositions generate exactly two different sorting networks: $S_1={\bf 121}$ and $S_2={\bf 212}$, connecting the identity element (corresponding to the vertex $(123)$\,) to the longest element (corresponding to the vertex $(321)$\,). 

Up to the obvious symmetries, we can construct two different sorting loops: $L_1=S_1S_2$ and $L_2=S_1S_1$, each having $2^6 = 64$ possible signatures. In a given loop, denote the first and last crossings of the $i$-th and $j$-th particles by $f(i,j)$ and $\ell(i,j)$. For the loop $L_2$ we have $\ell(i,j)+f(i,j)=7$, while for the loop $L_1$ we have that $\ell(i,j)-f(i,j)=3$. The linking number of the $i$-th and $j$-th strands only depends on the signatures of $f(i,j)$ and $\ell(i,j)$; in particular, $|\lk(i,j)|=\frac{1}{2}|\sigma_{f(i,j)}+\sigma_{\ell(i,j)}|$. It follows that out of $64$ possible signatures, only $8$ provide a braid with the vector $\vec \lk = \vec 0$. For the case of $n=3$, the vector $\vec\mu_3$ consists of one number: $\mu_3(1,2,3)$. We will now consider the two loops $L_1$ and $L_2$ separately.  

\subsection{\textbf{Sorting loop} $L_1$}
Out of $8$ possible signatures for the first half of the sorting loop, precisely $6$ provide $\mu_3(1,2,3) = 0$. The two signatures which provide a non-trivial braid are $(+-+)$ and $(-+-)$ (see Fig. \ref{Fig:2B}. The explanation for this lies in the inversion sets corresponding to each permutation.

Recall the definition of the inversion set:

\begin{dfn}
    For any permutation $p\in\symmetricn$ we can associate the \textit{inversion set}, a vector $\mathsf{inv}(p)\in \{\pm 1\}^N$ where $$\mathsf{inv}(p)_ {i,j}
    =\begin{cases}
\phantom{-}1&\mbox{ if } p(i)>p(j)\\
-1&\mbox{ if } p(i)<p(j)
\end{cases}.$$
\end{dfn}

Each permutation is uniquely determined by its inversion set (see Fig.\ref{fig:1B}). A permutation $p$ can be thought of as a new total ordering on the set of $n$ elements. Because of this, not all of the $2^N$ possible signatures will occur as an inversion set of some permutation, but only the $n!$ of these signatures corresponding to the $p$-determined order relation, $>_p$.

A total order such as the one we have defined must satisfy transitivity property: i.e., from $a>_p b>_p c$ it should follow that $a>_p c$. The signatures $(+-+)$ and $(-+-)$ listed above are precisely those which violate transitivity. Any other signature corresponds to a well-defined total order relation on the strands; for example the signature $(+--)$ corresponds to the ordering $3>1>2$ as shown in Fig. \ref{fig:1B}. Using this signature, we could draw the wiring diagram of $L_1$ on three parallel planes, each containing exactly one strand with no intersections.

In general for the signature given by the inversion set it is possible to establish a total ordering on the set of strands. The total ordering implies the  parallel planes construction from our example, and so the resulting $n!$ sorting braids are trivial. The two signatures $(+-+)$ and $(-+-)$ do not belong to the inversion set and so do not provide a complete ordering on the strands. In fact, the closures of the pure braids resulting from these signatures both create the Borromean rings (see Fig. \ref{Fig:2B}). Later on we will call such loop to be of \emph{type I}.

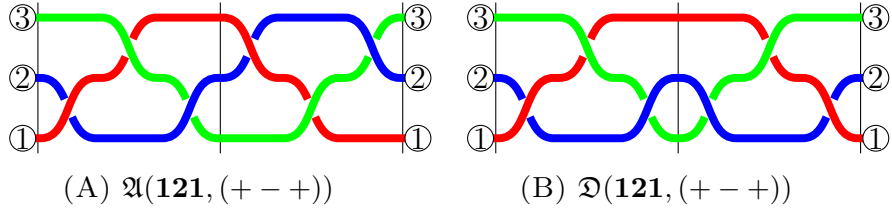
\begin{figure}[!hbt]\centering
\begin{subfigure}[b]{0.4\textwidth}
		\begin{tikzpicture}[scale=0.8]
		\braid[
		rotate=90,
		line width=3pt,
		floor command={%
			\draw (\floorsx,\floorsy) -- (\floorex,\floorsy);
		},
		style strands={1}{red},
		style strands={2}{blue},
		style strands={3}{green}
		] (braid) at (0,0) |s_1 s_2^{-1} s_1|s_2^{-1} s_1 s_2^{-1}|;
		\node[circle, draw, fill=white, inner sep=0 pt,at=(braid-3-e)] {\(3\)};
		\node[circle, draw, fill=white, inner sep=0 pt,at=(braid-3-s)] {\(3\)};
		\node[circle, draw, fill=white, inner sep=0 pt,at=(braid-2-e)] {\(2\)};
		\node[circle, draw, fill=white, inner sep=0 pt,at=(braid-2-s)] {\(2\)};
		\node[circle, draw, fill=white, inner sep=0 pt,at=(braid-1-e)] {\(1\)};
		\node[circle, draw, fill=white, inner sep=0 pt,at=(braid-1-s)] {\(1\)};
		\end{tikzpicture}
		\caption{$\asb({\bf 121},{(+-+)})$}\label{Fig:2B}
	\end{subfigure}
	\qquad
	\begin{subfigure}[b]{0.4\textwidth}
		\begin{tikzpicture}[scale=0.8]
		\braid[
		rotate=90,
		line width=3pt,
		floor command={%
			\draw (\floorsx,\floorsy) -- (\floorex,\floorsy);
		},
		style strands={1}{red},
		style strands={2}{blue},
		style strands={3}{green}
		] (braid) at (0,0) |s_1 s_2^{-1} s_1|s_1^{-1} s_2 s_1^{-1}|;
		\node[circle, draw, fill=white, inner sep=0 pt,at=(braid-3-e)] {\(3\)};
		\node[circle, draw, fill=white, inner sep=0 pt,at=(braid-3-s)] {\(3\)};
		\node[circle, draw, fill=white, inner sep=0 pt,at=(braid-2-e)] {\(2\)};
		\node[circle, draw, fill=white, inner sep=0 pt,at=(braid-2-s)] {\(2\)};
		\node[circle, draw, fill=white, inner sep=0 pt,at=(braid-1-e)] {\(1\)};
		\node[circle, draw, fill=white, inner sep=0 pt,at=(braid-1-s)] {\(1\)};
		\end{tikzpicture}
		\caption{$\dsb({\bf 121},{(+-+)})$}\label{Fig:2A}
	\end{subfigure}
	\caption{Sorting braids with zero vector of linking numbers.}\label{fig:sorting_braid}
\end{figure}

\subsection{\bf Sorting loop $L_2$.}
Our argument for $L_1$, it turns out, is independent of the sorting loop. The six signatures we mentioned above for $L_1$ also provide trivial braids for $L_2$. The two remaining signatures, however, act quite differently--they lead to trivial braids as well (see Fig. \ref{Fig:2A})! The reason for triviality is much simpler: the word $\mathbf{1\bar{2}1\bar{1}2\bar{1}}=\mathbf{1\bar{2}2\bar{1}}=\mathbf{1\bar{1}}=e$ reduces to the trivial one. 
Later we will refer to such loops as being of \emph{Type II.}

\section{The General Case.}
Now armed with some experience in the case of $n=3$, we proceed to examine the general case for $n$ particles. Recall the notation $N=\binom{n}{2}=\frac{n(n-1)}{2}$.

\subsection{The Asymptotics of linking numbers.}
We begin with a general statement on the signatures for a given sorting loop $ST$:

\begin{thm}
    Out of all $2^{2N}$ signatures for a sorting loop $ST$, there are exactly $\binom{N}{k} 2^N$ signatures for which $|\vec\lk| = k$.
\end{thm}

\begin{proof}
    For any pair of particles $(p_1,p_2)$ there exist exactly two indices corresponding to the crossing of their trajectories, namely $f(p_1,p_2)\in [1,N]$  and $\ell(p_1,p_2)\in [N, 2N]$. It follows that $\lk(p_1,p_2)=\frac 12 (\sigma_{f}+\sigma_{\ell})$. Furthermore, $\lk(p_1,p_2)=0$ if $\sigma_{f}=-\sigma_{\ell}$ and has magnitude $1$ otherwise. This implies that for any choice of the signatures $\sigma_1,\dots,\sigma_N$ of the first half of the loop, there exist exactly $\binom{N}{k}$ ways we may choose signatures on the second part to provide a non-zero vector of linking numbers.
\end{proof}

This theorem has the interesting effect of splitting the set of all signatures with respect to the magnitude of the signature's corresponding $\vec\lk$: $2^{2N}=2^N\sum \limits_{k=0}^N \binom{N}{k}$.

We can also compute the average linking number for a sorting loop over all possible signatures. The computation follows:
\[
\mathbb{E}(\lk)=2^{-2N}\sum \limits_{k=0}^{N} k \binom{N}{k} 2^N=2^{-N} \sum \limits_{k=0}^{N} k \binom{N}{k}=2^{-N} N 2^{N-1}=\frac{N}{2}.
\]

\subsection{Third order Milnor invariant.}
As it turns out, we can make many general statements about the $\mu_3$ invariants corresponding to signatures of much larger sorting networks. We begin with a definition:

\begin{dfn}
	For the sorting network $S={\bf j}_{1} {\bf j}_2\cdots 
	{\bf j}_N$ define the \textit{conjugate sorting network} as $S^*={\bf j}_{N}{\bf j}_{N-1}\cdots 
	{\bf j}_1$. 
\end{dfn}

Conjugation of a sorting network is a well-defined involution. Intuitively, $S^*$ corresponds to the sorting network that describes the path $S$ on the permutahedron, but read back from the long element to the identity. We can extend our definition of conjugation to apply to signed sorting networks as well:
$$({\bf j}_{1}^{\sigma_1} {\bf j}_2^{\sigma_2}\cdots 
{\bf j}_N^{\sigma_N})^*:={\bf j}_{N}^{-\sigma_N} {\bf j}_{N-1}^{-\sigma_{N-1}}\cdots 
{\bf j}_1^{-\sigma_1}.$$

From our construction, we can immediately deduce the following lemma:

\begin{lm} \label{lm:conjugate}
	For any signed sorting network $S(\sigma)$, the sorting braid $SS^*$ is trivial.
\end{lm}

\begin{proof}
The word ${\bf j}_{1}^{\sigma_1} {\bf j}_2^{\sigma_2}\cdots 
{\bf j}_N^{\sigma_N}{\bf j}_{N}^{-\sigma_N} {\bf j}_{N-1}^{-\sigma_{N-1}}\cdots 
{\bf j}_1^{-\sigma_1}$ is trivial since the product ${\bf j}_{N}^{-\sigma_N} {\bf j}_{N-1}^{-\sigma_{N-1}}\cdots 
{\bf j}_1^{-\sigma_1}$ is exactly the inverse of the product ${\bf j}_{1}^{\sigma_1} {\bf j}_2^{\sigma_2}\cdots 
{\bf j}_N^{\sigma_N}$.
\end{proof}

\subsection{$\mu_3$ invariants.}
According to the theorem by Tits, any two sorting networks are connected in $\sort(n)$ by a sequence of commutations $\bf ij\sim ji$, $|i-j|>1$ and braid moves $\bf jij\mapsto iji$, $|i-j|=1$ (see e.g. \cite{Reiner2013} and references therein). From the sequence of commutations we can construct the following definition:

\begin{dfn}
	For two sorting networks $S, T \in \sort(n)$, define $\langle S,T\rangle$ to be the number of braid moves in the minimal sequence connecting $S$ and $T$. Note that $\langle S,T\rangle=\langle T,S\rangle$.
\end{dfn}

This definition allows us to connect the construction of signed braids to the $n$-permutahedron. Similarly to before, let $\vec\mu_3$ be the vector of $\mu_3$ invariants for a particular sorting braid.

\begin{thm}\label{thm:Bruhat2Mu}
	For $S,T \in \sort(n)$, the loop $ST$ corresponds to $2^{\langle S,T\rangle}$ unlinked braids for which $|\vec\mu_3| \neq 0$.
\end{thm}

\begin{proof}
    The value $\langle S,T \rangle$ is by definition the minimal number of hexagonal faces bounded by the loop $ST$ on the $n$-permutahedron. Each hexagon represents a sorting 3-subbraid of type \textit{I}. As discussed in section \ref{sec:3}, exactly two signatures for every hexagon of type \textit{I} provide a non-zero $mu_3$. The theorem follows directly.
\end{proof}

Theorem \ref{thm:Bruhat2Mu} gives us a rough picture of the set of signatures providing non-trivial $\vec\mu_3$; however, it does not detect the finer details of the distribution of the magnitudes of $\vec\mu_3$ vectors. This is a consequence of $\mu_3(p,r,s)$ and $\mu_3(p,q,r)$ not being independent.

\subsection{Higher Milnor invariants.} As defined before, Milnor invariants can be generalized to an arbitrary number of strands to detect entanglement (or more precisely, Brunnian arrangements in the closure of a braid). Because of this, it is natural to ask what can be said about higher linking invariants for our sorting braids. To this end, we will prove the following theorem:

\begin{thm}\label{thm:main}
	If a sorting braid ST has all $\mu_3$ invariants equal to zero, the braid is trivial.
\end{thm}

The proof relies on the following lemma:

\begin{lm} \label{lm:path}
	If a sorting braid $ST$ has all $\mu_3$ invariants equal to zero, then it is isotopic to the braid $T^*T$. 
\end{lm}

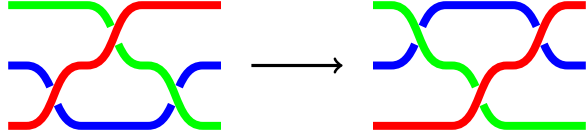
\begin{figure}[!hbt]\centering
	\begin{tikzpicture}[scale=0.8]
	\braid[
	rotate=90,
	line width=3pt,
	floor command={%
		\draw (\floorsx,\floorsy) -- (\floorex,\floorsy);
	},
	style strands={1}{red},
	style strands={2}{blue},
	style strands={3}{green}
	] (braid) at (0,0) s_1 s_2 s_1^{-1}; 
	\draw [->,very thick](4,1) -- (5.5,1);
\braid[
rotate=90,
line width=3pt,
floor command={%
	\draw (\floorsx,\floorsy) -- (\floorex,\floorsy);
},
style strands={1}{red},
style strands={2}{blue},
style strands={3}{green}
] (braid) at (0,-6) s_2^{-1} s_1 s_2; 
			\end{tikzpicture}
			\caption{Braid move for the signed braid}\label{fig:isotopy}
\end{figure}

\begin{proof}
    For our signed sorting network, define a commutation as ${\bf j}_1^{\sigma_1} {\bf j}_2^{\sigma_2}\to {\bf j}_2^{\sigma_2} {\bf j}_1^{\sigma_1}$ and a braid move as ${\bf j}_1^{\sigma_1} {\bf j}_2^{\sigma_2}{\bf j}_1^{\sigma_3}\to {\bf j}_2^{\sigma_3} {\bf j}_1^{\sigma_2}{\bf j}_2^{\sigma_1}$.
    
    Let $ST$ represent a sorting loop. By definition, there exists a path $S=S_0\to S_1\to\cdots\to S_m=T^*$ in $\sort(n)$ such that each step is either a braid move or a commutation. Since all $\mu_3$ are zero by assumption, the signature for every triple involved in the braid move must belong to the inversion set. Each of our braid moves, then, can be extended to the signed sorting networks $S(\sigma)$ and $T^*(\sigma')$ as well. This directly implies that the same minimal path leads from $S(\sigma)$ to $T^*(\sigma')$. As shown in figure \ref{fig:isotopy}, every step in the sequence is a braid isotopy.
\end{proof}

The claim of Theorem \ref{thm:main} now follows from lemmas \ref{lm:path} and \ref{lm:conjugate}.

\section{Two Particular Constructions.}
For a signed sorting network, we are free to choose several rules for defining the second half of a sorting loop. In this section, we will describe two particular constructions that come from seemingly natural choices of the second half.

\subsection{Algebraic sorting braids.} When considering sorting networks in terms of the $n$-permutahedron, a natural geometric approach to constructing a sorting loop is to approximate an equator. We will call the pure braid arising from this construction an \textit{algebraic sorting braid}, which we will define as follows:

\begin{dfn}
	An \emph{Algebraic sorting braid} (ASB) corresponding to the sorting 
	network 
	$S={\bf j}_1^{\sigma_1}\cdots {\bf j}_N^{\sigma_N}\in \sort(n)$ is 
	constructed as 
	\[\asb(S,\sigma)={\bf j}_1^{\sigma_1}\cdots {\bf j}_N^{\sigma_N} 
	({\bf n-j}_1)^{-\sigma_1}\cdots ({\bf n-j}_N)^{-\sigma_N}. \]
\end{dfn}

We can immediately prove that

\begin{lm}
	Every algebraic sorting braid is unlinked.
\end{lm}

\begin{proof}
	For any pair of particles $(p_1,p_2)$ such that $j_k=f(p_1,p_2)$ it follows that ${\bf j}_k$ corresponds to the crossing of the $p_1$-st and $p_2$-nd strands in the wiring diagram. It follows that $({\bf n-j}_k)$ in the product $({\bf n-j}_1)\cdots ({\bf n-j}_N)$ corresponds to the crossing of $(n+1-p_1)$-st and $(n+1-p_2)$-nd strands. These strands, though, are exactly those corresponding to the trajectories of the particles at the positions $(n+1-p_1)$ and $(n+1-p_2)$ after applying ${\bf j}_1\cdots {\bf j}_N$. Since the latter product corresponds to the sorting network, the particles are $p_1$ and $p_2$; therefore $\ell(p_1,p_2)=N+k$. Since $\sigma_{N-k}=-\sigma_k$, we have that $\lk(p_1,p_2)=\frac 12 (\sigma_{k}-\sigma_{k})=0$.
\end{proof}

We can also easily demonstrate the following:

\begin{thm}
	Out of all $2^{N}$ signatures of the sorting network $S$ there are exactly $n!$ which provide a trivial ASB.
\end{thm}

\begin{proof}
Any $m$-subbraid of an algebraic sorting braid is algebraic sorting braid on $m$ particles. Therefore, any $3$-subbraid is a sorting $3$-braid of type $I$ (see Fig. \ref{Fig:2B}). In order for the whole braid to be trivial, all $\mu_3$ invariants must to be zero. Therefore all the signatures corresponding to the $3$-braids must to belong to the inversion set and therefore provide an ordering on respective trajectories. Since any two trajectories intersect exactly once in the first half of a sorting braid, the orderings on two different triples $\{i,j,k\}$ and $\{i,j,k'\}$ must match at the intersection $\{i,j\}$. Hence, the order relation extends to a total order on the whole set of trajectories. Since there are exactly $n!$ different total orderings on the set of $n$ elements, the theorem follows.
\end{proof}

Since any triplet of strands in ASB form a sorting loop of type $I$ and signatures providing inversion sets for the triplets $\{i,j,k\}$ and $\{i,j,k'\}$ conditioned to the relative order of $i$ and $j$ are independent, the averaged magnitude of the third Milnor invariant of any ASB equals to $$\mathbb{E}|\mu_3(\asb(S))|=\frac 14 \binom{n}{3}.$$

\subsection{Dynamic sorting braid}
Another natural choice of the second half of a sorting loop is motivated by a dynamical interpretation of the word representation of the sorting network. We can construct a loop from a word on elementary transpositions simply by iterating the word a second time. Since this second part of the loop will act on the reversed interval, is seems natural to reverse the signatures of each transposition in the second part of the loop. We will define the resulting pure braid as follows:

\begin{dfn}
	A \emph{Dynamic sorting braid} (DSB) is obtained from the signed sorting 
	network $S(\sigma)={\bf j}_1^{\sigma_1}\cdots {\bf j}_N^{\sigma_N}$ by the 
	concatenation
	$$\dsb(S,\sigma)=S(\sigma) S({-\sigma})={\bf j}_1^{\sigma_1}\cdots {\bf j}_N^{\sigma_N}{\bf j}_1^{-\sigma_1}\cdots {\bf j}_N^{-\sigma_N}$$ (see 
	Fig. \ref{Fig:2A})).
\end{dfn}

Our definition of the DSB leads us to the following lemma:
\begin{lm}
	Let $f(p_1,p_2)$ be the index of the transposition corresponding to 
	the first crossing of the trajectories of the particles $p_1$ and $p_2$. If the signatures of the 
	transpositions $f(p_1,p_2)$ and $f(n+1-p_1,n+1-p_2)$ do not coincide, then $|\lk(p_1,p_2)|=1$. 
\end{lm}

\begin{proof}
    If $f(p_1,p_2)=k$, it follows that the transposition ${\bf j}_k$ in the word $S={\bf j}_1\cdots {\bf j}_N$ corresponds to the intersection of the trajectories of the particles $p_1$ and $p_2$. Hence, the $(N+k)$-th transposition in the sorting loop $SS$ corresponds to the  second intersection of the trajectories of the particles $(n+1-p_2)$ and $(n+1-p_1)$. By the definition of the dynamic sorting braid we have $\sigma_{N+k}=-\sigma_{k}$. Therefore, $|\lk(n+1-p_1,n+1-p_2)|=|\sigma_{f(n+1-p_1,n+1-p_2)}-\sigma_{k}|$. Since $k=f(p_1,p_2)$ the lemma follows.
\end{proof}

This lemma implies

\begin{thm}\label{thm:dynamic}
	Given $S\in\sort(n)$ there are $2^{\lfloor n/2\rfloor \lceil n/2\rceil}$ signatures $\sigma$ corresponding to the
	unlinked DSB. 
\end{thm}

\subsection{Slim networks.}
Unlike the case of the ASB, the amount of signatures providing a trivial DSB depends on the particular sorting network. It is natural to ask if there exists a network for which this amount attains its extremal values. This lead us to the following definition:

\begin{dfn}
	We will call a sorting network $S$ to be \emph{slim} if for any signature $\sigma$, the braid $\dsb(S,\sigma)$ is trivial. 
\end{dfn}

The simplest example of a slim sorting network is the construction $S_0=\bf (12\cdots (n-1))\,(1\cdot2\cdots (n-2))\,\cdots (12)\, 1 $. In fact, $S_0^2$ corresponds to the \emph{full twist} element in the braid group. We think it instructive to show for this particular example that we can achieve $S_0^*={\bf 1 (21)(321)\cdots ((n-1)\cdots 21)}$ from $S_0$ using only commutations. The process is iterative:

\begin{enumerate}
	\item Start with the second block $\bf (12\cdots (n-2))$ in $S_0$. The generator $\bf 1$ commutes with everything but the second letter in the block left to it. Therefore, we obtain
	$$\bf( 1(21) (34\cdots (n-1))(23\cdots (n-2))\cdots (12)1.$$
	Likewise, the generator $\bf 2$ in the second block commutes with everything but the $\bf 3$ on the left. We get 
	$$\bf (1(21)(32) (4\cdots (n-1)))(3\cdots (n-2))\cdots (12)1$$ etc. Finally we obtain
	 $$\bf (1(21)(32) (43)\cdots ((n-1)(n-2)))(12\cdots (n-3))\cdots (12)1.$$
	 \item Now perform the same operations with the block $\bf (12\cdots (n-3))$: Starting with the first generator, move it by commutations to the leftmost possible position, then move to the second generator and move it to the leftmost position achievable by commutations, etc. We get 
	 $$\bf(1(21)(321)(432)\cdots ((n-1)(n-2)(n-3)))(12\cdots(n-4))\cdots (12)1.$$ 
	 \item Continuing this procedure with each block we obtain $S_0^*$.
\end{enumerate}

This construction may be generalized to provide a family of slim sorting networks. However, we were unable to find a defining property of a sorting network to be slim obstructing us from the complete description of the set of slim networks. As such, we are left with two lingering questions about the nature of this set:

\begin{qst}
    How big is the set of slim sorting networks?
\end{qst}

One can also come up with the construction of somewhat opposite to the slim network, i.e. the network providing the DSB with the most magnitude of the vector of Milnor invariants (see Fig. \ref{fig:last}). Therefore we can define the \emph{fatness} of a sorting network as the maximal magnitude of $\vec\mu_3$ for $\dsb(S,\sigma)$ over all signatures $\sigma$. This measures the area on the permutahedron bounded by the loop $SS^*$, leading us to our second question:

\begin{qst}
	How ``fat'' is the typical sorting network, and more generally what does the distribution of fatness look like across all sorting networks?
\end{qst}

\subsection{Experimental results.}
According to a theorem by Stanley, the cardinality of the set $\sort(n)$ grows immensely fast.

\begin{thm}[\cite{Stanley1986}]
	Set $\sort(n)$ has exactly 
	\[\frac{\binom{n}{2}!}{1^{n-1} 3^{n-2} 5^{n-3}\cdots (2n-3)^1}\] 
	elements.
\end{thm}

Because of this, we were only able to directly analyze results for the first few values of $n$. Below we examine the set $\sort(4)$ in detail, and we provide numerical results obtained for the set $\sort(5)$.

\subsubsection{The case of $n = 4$}
The set $\sort(4)$ contains sixteen distinct sorting networks, which form four equivalence classes up to obvious symmetries:
$$\begin{aligned}
S_0=\bf 1 2 3 1 2 1,\\S_1=\bf 1 23212,\\S_2=\bf 1 3 2 3 12,\\S_3=\bf 1 3 2 1 3 2.\end{aligned}$$
From theorem \ref{thm:dynamic}, it follows that sixteen signatures for each of these sorting networks provides an unlinked DSB. Among them, $S_0$ is a slim network for which no signatures provide a non-trivial linked braid (see Fig. \ref{fig:S0}). 

\begin{figure}[hbt]
	\centering
	\begin{subfigure}[b]{0.24\textwidth}
		\includegraphics[width=\textwidth]{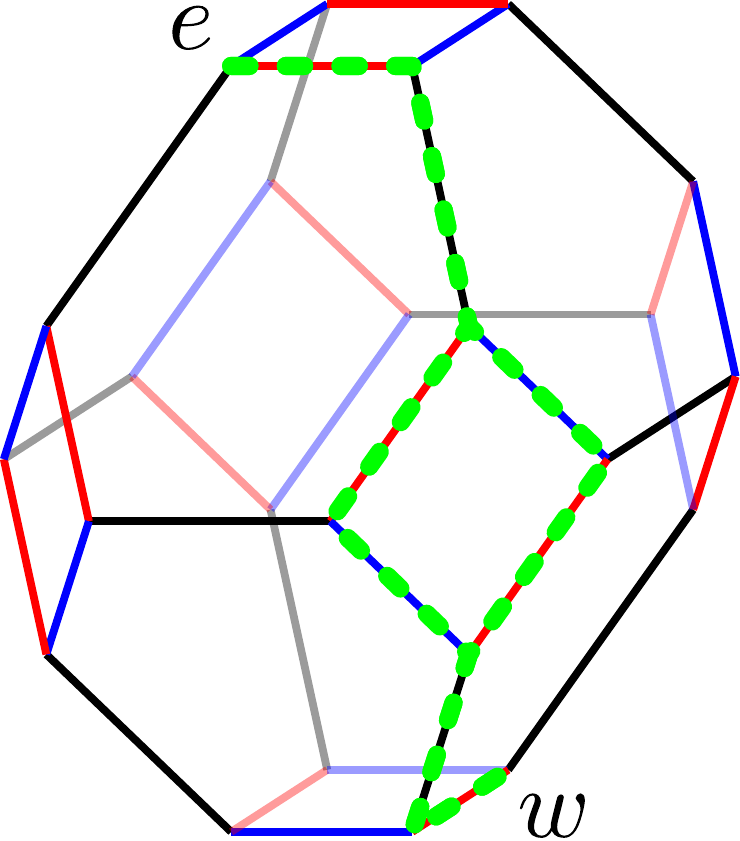}
		\caption{$S_0$}\label{fig:S0}
	\end{subfigure} \begin{subfigure}[b]{0.24\textwidth}
		\includegraphics[width=\textwidth]{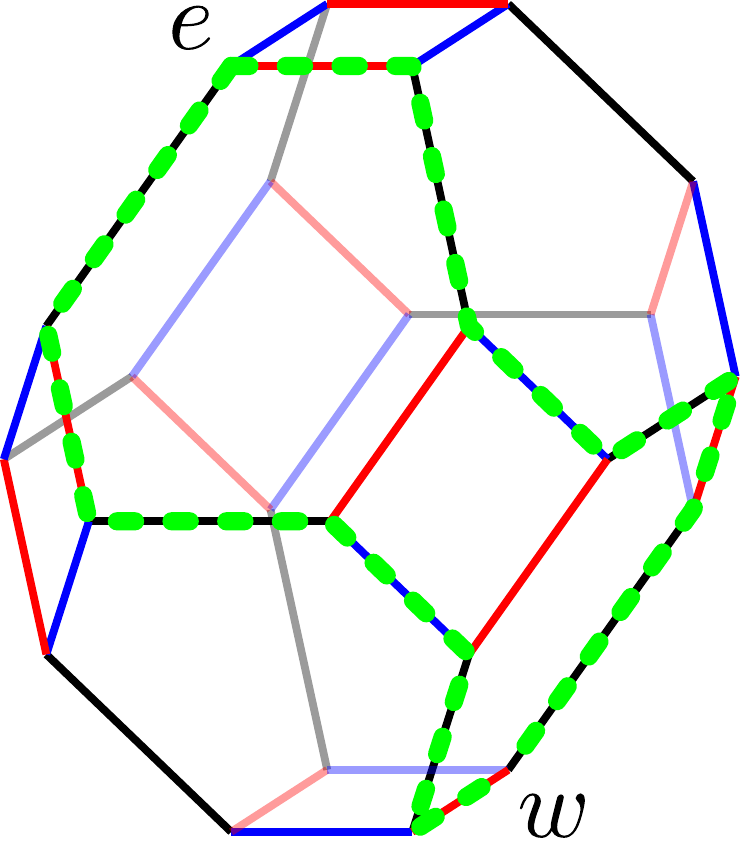}
		\caption{$S_1$} \label{fig:S1}
	\end{subfigure}
	\begin{subfigure}[b]{0.24\textwidth}
		\includegraphics[width=\textwidth]{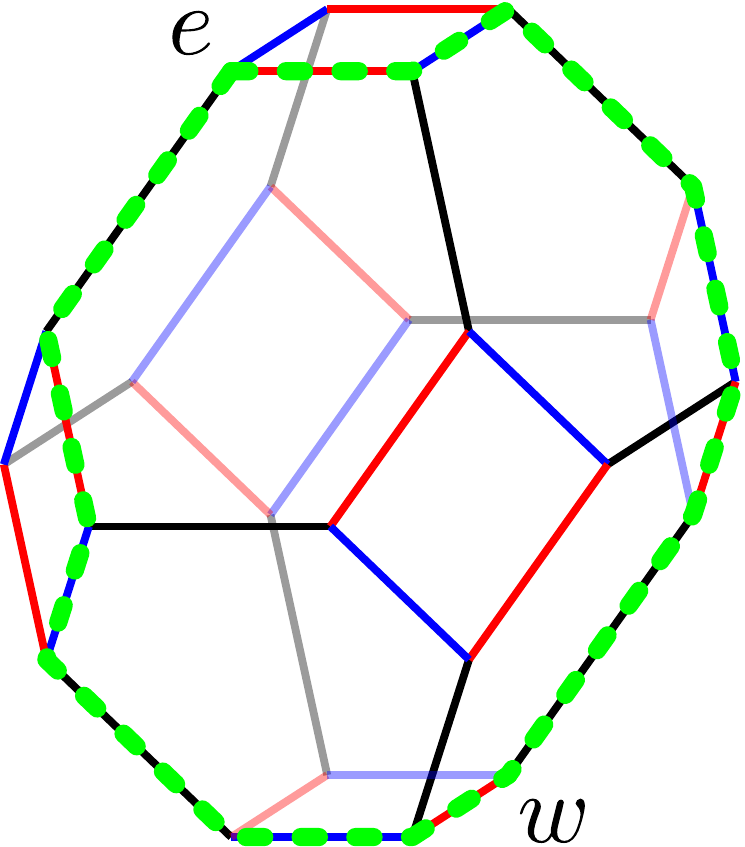}
		\caption{$S_2$}\label{fig:S2}
	\end{subfigure}\begin{subfigure}[b]{0.24\textwidth}
		\includegraphics[width=\textwidth]{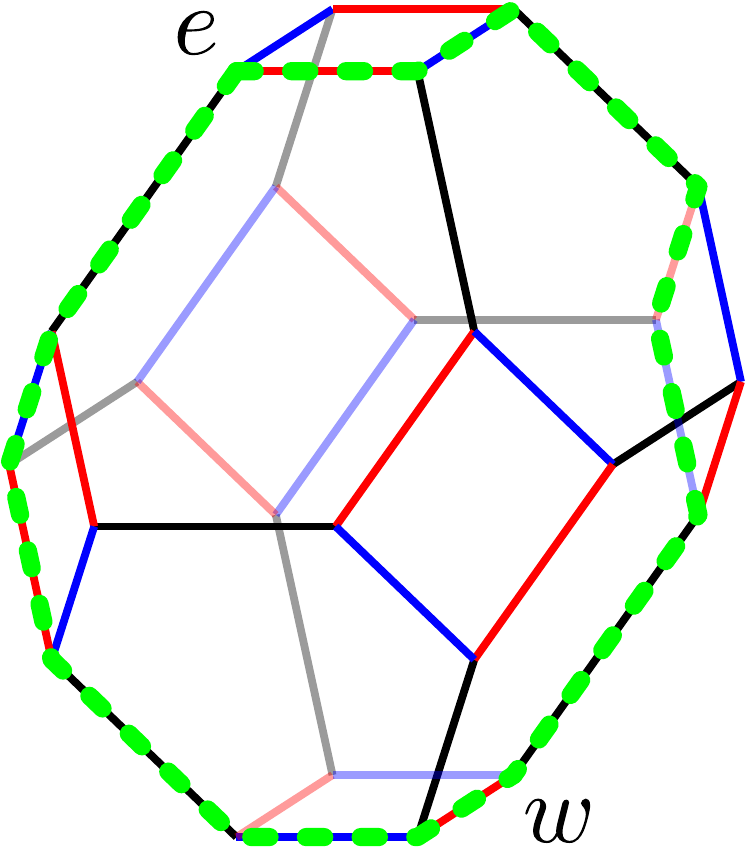}
		\caption{$S_3$}\label{fig:S3}
	\end{subfigure}
	\caption{Loops on $4$-permutohedron, corresponding to the Dynamical Sorting Braids.}
\end{figure}

The DSB obtained from the sorting network $S_1$ encircles two hexagons (see Fig. \ref{fig:S1}). In other words, $S_1^*$ can be obtained from $S_1$ by two braid moves: $\bf 123212\mapsto 123121 \sim 121321\mapsto 212321$, where  $\sim$ denotes commutations and $\mapsto$ denotes braid moves. It follows from theorem \ref{thm:Bruhat2Mu} that there are $4$ signatures which provide a nontrivial unlinked DSB. Interestingly, each of these braids has $|\vec\mu_3|=2$. For example, $\bf{1\bar{2}32\bar{1}2}$ provides the pure braid with $|\mu_3(1,2,3)|=|\mu_3(2,3,4)|=1$ (see figure \ref{fig:last}). The expected $\mu_3$ invariant for a signed DSB based on $S_1$ is thus $1/2$.

\begin{figure}[!hbt]\centering
		\begin{tikzpicture}[scale=0.8]
		\braid[
		rotate=90,
		line width=3pt,
		floor command={%
			\draw (\floorsx,\floorsy) -- (\floorex,\floorsy);
		},
		style strands={1}{red},
		style strands={2}{blue},
		style strands={3}{green},
		style strands={4}{gray}
		] (braid) at (0,0) |s_1 s_3^{-1} s_2 s_1 s_3^{-1} s_2|s_1^{-1} s_3 s_2^{-1} s_1^{-1} s_3 s_2^{-1}|;
		\node[circle, draw, fill=white, inner sep=0 pt,at=(braid-4-e)] {\(4\)};
		\node[circle, draw, fill=white, inner sep=0 pt,at=(braid-4-s)] {\(4\)};
		\node[circle, draw, fill=white, inner sep=0 pt,at=(braid-3-e)] {\(3\)};
		\node[circle, draw, fill=white, inner sep=0 pt,at=(braid-3-s)] {\(3\)};
		\node[circle, draw, fill=white, inner sep=0 pt,at=(braid-2-e)] {\(2\)};
		\node[circle, draw, fill=white, inner sep=0 pt,at=(braid-2-s)] {\(2\)};
		\node[circle, draw, fill=white, inner sep=0 pt,at=(braid-1-e)] {\(1\)};
		\node[circle, draw, fill=white, inner sep=0 pt,at=(braid-1-s)] {\(1\)};
		\end{tikzpicture}
		\caption{$\dsb(S_3,(+-++-+))$ provides $|\mu_3(1,2,3)|=|\mu_3(1,3,4)|=1$.}\label{fig:last}
	\end{figure}
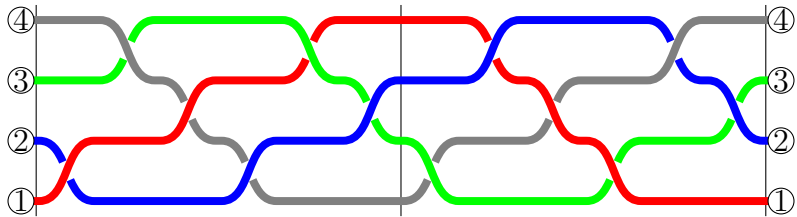

For both $S_2$ and $S_3$, the corresponding sorting loop encircles four hexagons (see Figs \ref{fig:S2} and \ref{fig:S3}). Theorem \ref{thm:Bruhat2Mu} ensures the existence of $16$ different signatures leading to the non-trivial braid. However, not all these signatures correspond to the dynamical sorting braid. Recall, that DSB definition implies symmetry conditions on the signatures on the first and second halves of the sorting loop. For instance, contrary to $64$ possible signatures providing unlinked braid in general setting, only $16$ of them provide unlinked DSB (see theorem \ref{thm:dynamic}). Out of these $16$ there are $8$ providing nontrivial unlinked DSB, with $|\vec\mu_3| = 2$. The expected $\mu_3$ invariant for these networks is $1$, so the expected invariant over all unlinked DSB is $5/8$.

\subsubsection{The case of $n = 5$}
$\sort(5)$ contains $768$ distinct sorting networks, each having $64$ signatures providing an unlinked DSB. Below, we present a distribution of the magnitudes of $\vec\mu_3$ averaged over all signatures for each network. These averaged quantities happen to belong to the set $\{0,1/2,1,3/2,2\}$ with majority being $0$. The expected $\mu_3$ invariant for $\sort(5)$ is $113/128$.
\begin{figure}[!hbt]
	\centering
	\includegraphics[width=0.9\textwidth]{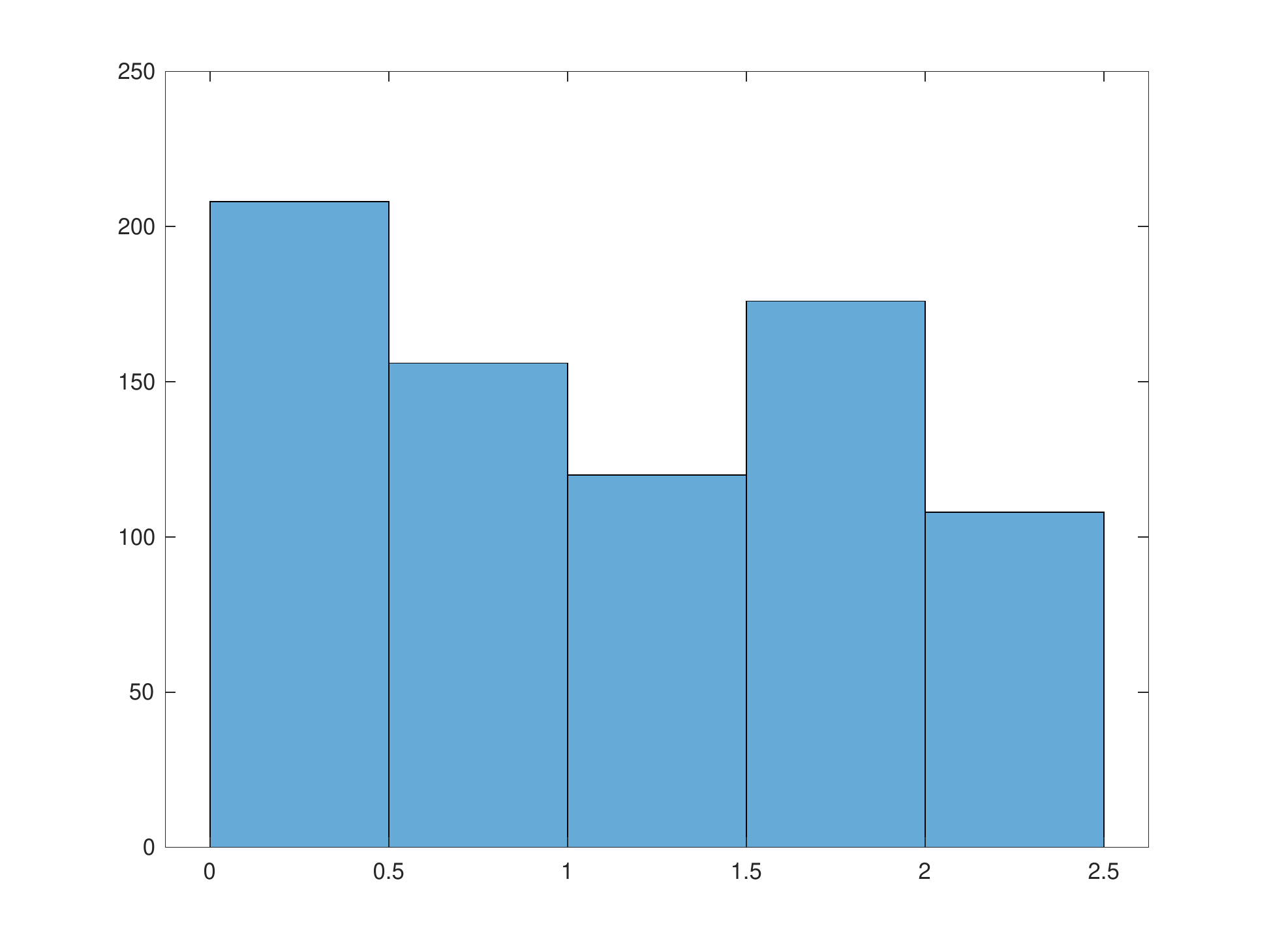}
	\caption{Distribution of the magnitudes of the third Milnor invariants among all $768$ unlinked DSB on $5$ strands.}
	\label{fig:my_label}
\end{figure}

 The experimental results lead us to numerous conjectures. In particular, it seems plausible that: \begin{itemize}
	\item Set of slim networks occupies non-zero measure in the set of all sorting networks, i.e. $\frac{|\mbox {slim}|}{|\sort(n)|}\to \varepsilon\ne 0$ as $n\to \infty$.
	\item Expected magnitude of the vector of Milnor invariants tends to some constant $\mathbb{E} |\mu_3|\to \mu<\infty$ as $\nLeftarrow\to \infty$.
\end{itemize}

\subsection*{Acknowledgements.} The authors are thankful to N. Stein for the beautiful COCALC environment where most of our numerical studies were conducted, as well as to N. Williams for deep and fruitful discussions. Part of work of C.K. was completed during his summer spent participating in Budapest Semesters in Mathematics.  
    \bibliographystyle{alpha}
    \bibliography{sorting}

\end{document}